\documentclass[11pt]{amsart}
\usepackage{amssymb, enumerate, color}

\def\Q{\mathbb{Q}}
\def\P{\mathbb{P}}
\def\A{\mathbb{A}}

\DeclareMathOperator{\lcm}{lcm}

\newcommand{\col}{\,{:}\,}
\newcommand{\tth}{^{\operatorname{th}}}
\newcommand{\sst}{^{\operatorname{st}}}

\theoremstyle{plain}
\newtheorem{thm}{Theorem}
\newtheorem{lem}{Lemma}

\theoremstyle{definition}
\newtheorem*{defn}{Definition}
\newtheorem{conj}{Conjecture}
\newtheorem{exmp}{Example}
\theoremstyle{remark}
\newtheorem*{rem}{Remark}

\newtheorem{case}{Case}

\begin{document}
\title[Polynomial Morphisms on Projective Space]{Rational Periodic Points for Degree Two Polynomial Morphisms on Projective Space}

\author[Hutz]{Benjamin Hutz}
\address{\newline
Department of Mathematics and Computer Science,\newline
Amherst College, 
P.O. 5000,\newline
Amherst, MA 01002-5000,
USA
}
\email{bhutz@amherst.edu}



\date{\today}

\maketitle

\begin{abstract}
	This article addresses the existence of $\Q$-rational periodic points for morphisms of projective space.  In particular, we construct an infinitely family of morphisms on $\P^N$ where each component is a degree $2$ homogeneous form in $N+1$ variables which has a $\Q$-periodic point of primitive period $\frac{(N+1)(N+2)}{2} + \left\lfloor \frac{N-1}{2}\right\rfloor$.  This result is then used to show that for $N$ large enough there exists morphisms of $\P^N$ with $\Q$-rational periodic points with primitive period larger that $c(k)N^k$ for any $k$ and some constant $c(k)$.
\end{abstract}


\section{Introduction and Statement of Results}
    Let $\phi:\P^1 \to \P^1$ be a morphism of degree 2 which has a totally ramified fixed point at infinity, in other words, a \emph{polynomial morphism}.  We will denote $\phi^n$ as the $n\tth$ iterate of $\phi$.  A point $P \in K$ is called \emph{periodic of period $n$ for $\phi$} if there is a positive integer $n$ such that $\phi^n(P)=P$.  If $n$ is the smallest such integer, it is called the \emph{primitive period of $P$ for $\phi$}.  Northcott's theorem \cite{Northcott} tells us that $\phi$ can have only finitely many rational periodic points defined over a number field and, hence, the primitive periods of rational periodic points must be bounded.  For $n=1$, $2$, and $3$ there are infinitely many examples of degree 2 polynomial maps defined over $\Q$ with $\Q$-rational periodic points with primitive period $n$.  Morton \cite{Morton2} showed that there are no such maps with $\Q$-rational primitive $4$-periodic points.  Flynn, Poonen, and Schaefer \cite{FPS} showed that there are no such maps with $\Q$-rational primitive $5$-periodic points and made the following conjecture:
    \begin{conj} \label{FPS_conj}
        For $n \geq 4$ there is no quadratic polynomial $f \in \Q[x]$ with a rational periodic point with primitive period $n$.
    \end{conj}
    More recently, Stoll \cite{Stoll} has shown conditionally that there are no degree 2 polynomial maps with $\Q$-rational primitive $6$-periodic points.  For degree two rational maps, Manes \cite[Theorem 4]{Manes2} shows the existence of maps with $\Q$-rational periodic points of primitive period $4$ and provides evidence for no maps with $\Q$-rational points of primitive period $5$ or $6$.  This article examines the possible primitive period of a $\Q$-rational periodic point for a degree two polynomial morphism on $\P^N$ defined over $\Q$.
    \begin{defn}
        We define a \emph{polynomial} map $\phi:\P^N \to \P^N$ of degree $d$ with coordinates $[x_0,\ldots,x_N]$ to be a map defined as
        \begin{equation*}
            \phi(x_0,\ldots,x_N) = [\phi_0(x_0,\ldots,x_N),\ldots,\phi_{N-1}(x_0,\ldots,x_N),x_N^d],
        \end{equation*}
        where each $\phi_i$ is a homogeneous form of degree $d$ in the variables $x_0,\ldots,x_N$.  Such a map is a \emph{morphism} if $\phi_1,\ldots, \phi_{N-1}$ have no nontrivial common zeroes when $x_N=0$.
    \end{defn}
\subsection{A First Example}
    Consider a degree two polynomial map $\phi:\P^1 \to \P^1$ given by
    \begin{equation*}
        \phi(x,y) = [ax^2 + bxy + cy^2, y^2].
    \end{equation*}
    We wish to find constants $a$, $b$, and $c$ such that $P=[0,1]$ is a periodic point of primitive period $3$ for $\phi$.  To do so we choose any two distinct other points $P_1$ and $P_2$ and solve the three linear equations $\phi(P) = P_1$, $\phi(P_1) = P_2$, and $\phi(P_2)=P$ in the three unknowns $a$, $b$, and $c$ to find a suitable map $\phi$.  Since we will use a similar, albeit more complicated, construction in Theorem \ref{thm_main}, we explore this construction in detail for this simple example.

    We see that $$\phi([0,1]) = [c,1],$$ so we choose $c \neq 0$, say $c=1$.  Then we have $$\phi([1,1]) = [a+b+1,1].$$  We now choose $b$ so that $$\phi([1,1]) \neq [0,1] \text{ or } [1,1],$$ say $b=1-a$.  Then, we have $\phi([1,1]) = [2,1]$ and so $$\phi([2,1]) = [2a+3,1].$$  Finally, we choose $a=-3/2$ to have $\phi([2,1]) = [0,1]$, making the degree two polynomial map $$\phi([x,y]) = [-3/2x^2+5/2x+1,y^2]$$ have $[0,1]$ as a primitive $3$-periodic point.

    Trying to construct a $\Q$-rational primitive 4-periodic point in the same manner, at a minimum, requires more care.  The obstruction lies in having to solve four equations in three unknowns.  Conjecture \ref{FPS_conj} states that for a degree 2 polynomial morphism on $\P^1$ it is never possible to solve these larger systems of equations and the number of coefficients $\binom{2+2}{2} = 3$ is an upper bound on the primitive $\Q$-rational periods.

    Theorem \ref{thm_main} demonstrates an infinite family of polynomial maps on $\P^N$ with periodic points with primitive period larger than $\binom{N+2}{2}$.  Theorem \ref{thm2} shows that these families contain infinitely many maps which are in fact morphisms of $\P^N$.  Theorem \ref{thm3} uses Theorem \ref{thm_main} and Theorem \ref{thm2} to show that the primitive period of $\Q$-rational periodic points for polynomial morphisms of $\P^N$ grows faster than $c(k)N^k$ for any $k$ and some constant $c(k)$.

\subsection{The Main Results}
    In general, we can construct a degree $2$ polynomial map with a $\Q$-rational periodic point with primitive period equal to the number of coefficients of a quadratic form in $N+1$ variables $\binom{N+2}{2} = \frac{(N+1)(N+2)}{2}$, by choosing one coefficient with each successive iterate as we did above for $\P^1$.  We show that for $N \geq 2$ we can construct polynomial maps on $\P^N$ with a periodic point with primitive period larger than this value.
    \begin{thm} \label{thm_main}
        Let $N \geq 2$.  There is an infinite family of degree two polynomial maps $\phi:\P^N \to \P^N$ with a $\Q$-rational periodic point with primitive period
        \begin{equation*}
            \begin{cases}
                \leq 7=\frac{(N+1)(N+2)}{2} + 1 & \text{for } N=2\\
                \leq \frac{(N+1)(N+2)}{2} + \left\lfloor \frac{N-1}{2}\right\rfloor & \text{for } N \geq 3,
            \end{cases}
        \end{equation*}
        where $\lfloor x \rfloor$ denotes the greatest integer less than or equal to $x$.
        Moreover, the dimension of the family is at least $N$.
    \end{thm}
    \begin{thm} \label{thm2}
        The infinite family of maps constructed in Theorem \ref{thm_main} contains infinitely many morphisms.
    \end{thm}
    \begin{thm} \label{thm3}
        For $N$ large enough, there exists a degree two polynomial morphism of $\P^N$ with a $\Q$-rational periodic point with primitive period larger than $c(k)N^k$ for any $k$ and some constant $c(k)$ depending on $k$.
    \end{thm}
    In general, the bounds in Theorem \ref{thm_main} are not upper bounds on the primitive period.  Several examples of polynomial morphisms with $\Q$-rational points with larger primitive period are included at the end of the article.

\section{Proof of Theorem \ref{thm_main}}
    We denote the $i\tth$ coordinate of a point $P \in \P^N$ as $x_i(P)$ and denote the polynomial map $\phi:\P^N \to \P^N$ as
    \begin{equation*}
        x_i(\phi(x_0,\ldots,x_N)) =
            \begin{cases}
                \sum_{j=0}^{N-1}\sum_{k=j}^{N}c_i(j,k)x_jx_k & \text{for } i=0,\ldots,N-1\\
                c_N(N,N)x_N^2 & \text{for } i = N.
            \end{cases}
    \end{equation*}
    We denote the $n\tth$ image of $P$ by $\phi$ as $\phi^n(P)=P_n$.

    The method of construction is to choose appropriate values of the constants $c_i(j,k)$ so that the coordinates of each iterate are linear in at most two of the $c_i(j,k)$.  When we have chosen all of the $c_i(j,k)$ with $j \neq k$, we will then be able to choose $c_i(j,j)$ so that $\phi(P)$ is determined and $\phi(\phi(P))$ is linear in one of the $c_i(j,j)$, allowing the primitive period to increase beyond the trivial value $\binom{N+2}{2}$ determined by the number of coefficients.  We treat the case of $N=2$ separately.

    In Lemma \ref{lem1}, we choose the initial sequence of images by specifying $c_i(j,k)$ $0 \leq i \leq N$ for one pair $(j,k)$ for each image.  Then we proceed with the construction to increase the primitive period beyond the trivial lower bound.
    \begin{lem} \label{lem1}
        Let $\phi:\P^N \to \P^N$ be a degree two polynomial map.  We may choose the first $(N^2+N)/2-1$ images of $[0,\ldots,0,1]$ as
        \begin{align*}
            &[0,\ldots,0,1] \xrightarrow{\phi} [1,0,\ldots,0,1] \xrightarrow{\phi} [0,1,0,\ldots,0,1] \xrightarrow{\phi} \cdots \xrightarrow{\phi} [0,\ldots,0,1,1]\\
            &\xrightarrow{\phi} [1,1,0,\ldots,0,1] \xrightarrow{\phi} [0,1,1,0,\ldots,0,1] \xrightarrow{\phi} \cdots \xrightarrow{\phi} [0,\ldots,0,1,1,1]\\
        &\xrightarrow{\phi} [1,1,1,0,\ldots,0,1] \xrightarrow{\phi} [0,1,1,1,0,\ldots,0,1] \xrightarrow{\phi} \cdots \xrightarrow{\phi} [0,\ldots,0,1,1,1,1]\\
        &\vdots\\
            &\xrightarrow{\phi}[1,\ldots,1,0,1] \xrightarrow{\phi} [0,1,\ldots,1,1] \xrightarrow{\phi} [1,\ldots,1]
        \end{align*}
        by choosing all of the $c_i(j,k)$ except $c_i(0,N-1)$ and $c_i(k,k)$ for each $0 \leq k \leq N-1$ and each $0 \leq i \leq N-1$.  Furthermore, $x_i(\phi([1,\ldots,1]))$ is of the form $$a_ic_i(0,N-1) + b_i$$ for some constants $a_i,b_i$ for all $0 \leq i \leq N-1$.
    \end{lem}
    \begin{proof}
        Let $P=[0,\ldots,0,1]$.  Then $x_i(\phi(P))$ is linear in $c_i(N,N)$, so we can choose
        \begin{equation*}
            \phi(P)= [1,0,\ldots,0,1]
        \end{equation*}
        by choosing
        \begin{equation*}
            c_i(N,N) = \begin{cases}
              1 & \text{for } i=0 \text{ and } i=N\\
              0 & \text{for } 1 \leq i \leq N-1.
            \end{cases}
        \end{equation*}
        Next we choose the sequence of points
        \begin{align*}
            [1,0,\ldots,0,1] &\xrightarrow{\phi} [0,1,0,\ldots,0,1] \xrightarrow{\phi} \cdots \\
            &\xrightarrow{\phi} [0,0,\ldots,0,1,1] \xrightarrow{\phi} [1,1,0,\ldots,0,1],
        \end{align*}
        where
        \begin{equation*}
            [1,1,0,\ldots,0,1] = P_{N+1}.
        \end{equation*}
        We can do this because $x_i(\phi(P_j))$ is of the form
        \begin{equation*}
            x_i(\phi(P_j)) =
            \begin{cases}
                c_i(j-1,j-1) + c_i(j-1,N) + 1 & \text{for } i=0\\
                c_i(j-1,j-1) + c_i(j-1,N) & \text{for } 1 \leq i \leq N-1\\
                1 & \text{for } i =N.
            \end{cases}
        \end{equation*}
        We choose
        \begin{equation*}
            c_i(j-1,N) =
            \begin{cases}
                -1-c_i(j-1,j-1) & \text{for } i=0\\
                1 - c_i(j-1,j-1) & \text{for } i = j-1\\
                -c_i(j-1,j-1) & \text{otherwise}.
            \end{cases}
        \end{equation*}
        \begin{rem}
            With these choices of $c_i(N,N)$ and $c_i(k,N)$ for all $0 \leq k \leq N-1$, the image $x_k(\phi(x_0,x_1,\ldots,x_{N-1},1))$ contains terms of the form $c_k(i,i)x_i^2 - c_k(i,i)x_i$ for all $0 \leq i \leq N$ and $0 \leq k <N-1$.  Consequently, if $x_i=1$ or $x_i=0$, then $c_k(i,i)$ does not appear in the $k\tth$ coordinate of the image.
        \end{rem}
        Next we choose the sequence of images
        \begin{align*}
            [1,1,0,\ldots,0,0,1] &\xrightarrow{\phi} [0,1,1,0,\ldots,0,1] \xrightarrow{\phi} \cdots \\
            &\xrightarrow{\phi} [0,\ldots,0,1,1,1] \xrightarrow{\phi} [1,1,1,0,\ldots,0,1]
        \end{align*}
        until
        \begin{equation*}
            P_{2N-1} = [1,1,1,0,\ldots,0,1].
        \end{equation*}
        We can do this because we have already chosen $c_i(k,N)$ for all $0 \leq k \leq N-1$ and $c_i(N,N)$, causing the $i\tth$ coordinate for all $0 \leq i \leq N-1$ of each image in this sequence to be linear only in the single coefficient $c_i(k,k+1)$ for some $0 \leq k \leq N-2$.

        Next we choose the sequence of images
        \begin{align*}
            [1,1,1,0,\ldots,0,1] &\xrightarrow{\phi} [0,1,1,1,0,\ldots,0,1] \xrightarrow{\phi} \cdots \\
            &\xrightarrow{\phi} [0,\ldots,0,1,1,1,1]  \xrightarrow{\phi} [1,1,1,1,0,\ldots,0,1].
        \end{align*}
        Since we have already chosen $c_i(N,N)$, $c_i(k,N)$ for all $0 \leq k \leq N-1$, and $c_i(k,k+1)$ for all $0 \leq k \leq N-2$ and $0 \leq i \leq N-1$, the $i\tth$ coordinate of these iterates is linear in $c_i(k,k+2)$ for some $0 \leq k \leq N-3$.

        We repeat this process until we have
        \begin{align*}
            \phi(P_{(N^2+N)/2 -3}) &= [1,\ldots,1,0,1,1] \quad  (\text{linear in }c_i(0,N-2))\\
            \phi(P_{(N^2+N)/2 -2}) &= [0,1,\ldots,1]\quad  (\text{linear in }c_i(1,N-1))\\
            \phi(P_{(N^2+N)/2 -1}) &= [1,\ldots,1] \quad  (\text{linear in }c_i(0,N-1)).
        \end{align*}
        The only coefficients not yet chosen are $c_j(i,i)$ for all $0 \leq i \leq N-1$ and $c_j(0,N-1)$.  We also know that $x_i(\phi(P))$ is linear in $c_i(0,N-1)$ for all $0 \leq i \leq N-1$.
    \end{proof}
    We are now ready to increase the primitive period beyond the trivial lower bound.
    \begin{proof}[Proof of Theorem \ref{thm_main}]
        \begin{case}($N=2$)

            We begin where Lemma \ref{lem1} finished.  We have chosen $c_i(2,2)$ for all $0 \leq i \leq 2$ and $c_i(0,2)$ and $c_i(1,2)$ for all $0 \leq i \leq 1$ to have the sequence of points
            \begin{equation*}
                [0,0,1] \xrightarrow{\phi} [1,0,1] \xrightarrow{\phi} [0,1,1] \xrightarrow{\phi} [1,1,1].
            \end{equation*}
            Now we choose $c_i(0,N-1) = c_i(0,1)$ so that
            \begin{equation*}
                \phi([1,1,1]) = [0,K(0,1),1]
            \end{equation*}
            for some constant $K(0,1) \not\in \{0,1\}$.  Each coordinate $x_i(\phi([0,K(0,1),1]))$ is linear in $c_i(1,1)$.  We then choose the $c_i(1,1)$ so that
            \begin{equation*}
                \phi([0,K(0,1),1])) = [0,K(1,1),1],
            \end{equation*}
            where $K(1,1)$ is a constant with $K(1,1) \not\in  \{0,1,K(0,1)\}$.  So we have three points chosen of the form $[0,x_1,1]$ whose images are given by
            \begin{align*}
                [0,0,1] &\xrightarrow{\phi} [1,0,1]\\
                [0,1,1] &\xrightarrow{\phi} [1,1,1]\\
                [0,K(0,1),1] &\xrightarrow{\phi} [0,K(1,1),1].
            \end{align*}
            Since $x_i(\phi[0,x_1,1])$ is a quadratic polynomial in $x_1$ for $0 \leq i \leq 1$, the image $\phi([0,K(1,1),1])$ is determined by these three known points and is of the form
            \begin{equation*}
                \phi([0,K(1,1),1]) = [k_0,k_1,1]
            \end{equation*}
            for some constants $k_0$ and $k_1$.  Note that we may need to exclude finitely many choices of $K(0,1)$ and $K(1,1)$ (and hence of $c_1(0,1)$ and $c_1(1,1)$) so that $k_0 \not\in \{0,1\}$.  So we have that each $x_i(\phi([k_0,k_1,1]))$ is linear in $c_i(0,0)$ and we choose the $c_i(0,0)$ so that
            \begin{equation*}
                \phi(\phi([0,K(1,1),1])) = [0,0,1].
            \end{equation*}
            This is a primitive $7$-periodic point, and the family is dimension $2$ since we can choose $c_1(0,1)$ and $c_1(1,1)$ arbitrarily (with finitely many exceptions).

            It is easy to modify this construction to get points with primitive periods $1,\ldots, 6$ because at each stage we are linear in at most two variables.  So we simply choose the constant so that $\phi^n(P) = [0,0,1]$ at the appropriate iterate.  The dimension of these families is larger since there are more free coefficients.
        \end{case}
        \begin{case} ($N \geq 3$)

            We begin where Lemma \ref{lem1} finished and choose the $c_i(0,N-1)$ so that we have $$\phi(P_{(N^2+N)/2})=[0,K(0,N-1),1,\ldots,1]$$ where $K(0,N-1) \not\in \{0,1\}$ is a constant.  Each coordinate $x_i(\phi([0,K(0,N-1),1,\ldots,1]))$ is linear in $c_i(1,1)$.
            Choose the $c_i(1,1)$ to have
            \begin{equation*}
                \phi([0,K(0,N-1),1,\ldots,1]) = [K(1,1),1,\ldots,1,0,1],
            \end{equation*}
            where $K(1,1) \not\in\{0,1\}$ is some constant.  Now we have that the $i\tth$ coordinate of the image of $[K(1,1),1,\ldots,1,0,1]$ is linear in $c_i(0,0)$ for all $0 \leq i \leq N-1$.
            So we choose the $c_i(0,0)$ so that
            \begin{equation*}
                \phi([K(1,1),1,\ldots,1,0,1]) = [0,K(0,0),1,\ldots,1]   
            \end{equation*}
            for some constant $K(0,0) \not\in \{0,1,K(0,N-1)\}$.
            Note that there are three points of the form $[0,x_1,1,\ldots,1]$ whose images are given by
            \begin{align*}
                [0,1,1,\ldots,1] &\xrightarrow{\phi} [1,\ldots,1]\\
                [0,K(0,N-2),1,\ldots,1] &\xrightarrow{\phi} [K(1,1),1,\ldots,1,0,1]\\
                [0,0,1,\ldots,1] &\xrightarrow{\phi} [1,1,\ldots,1,0,1].
            \end{align*}
            Since each $x_i(\phi([0,x_1,1,\ldots,1]))$ is a degree 2 polynomial in $x_1$, the three known points and their images completely determine the image of any point of the form $[0,x_1,1,\ldots,1]$.  The $0\tth$ coordinate is a non-constant function of $x_1$ and the $(N-1)\sst$ coordinate is a non-constant function of $x_1$.  The remaining coordinates take on a constant value of $1$.  Therefore, we have
            \begin{equation*}
                \phi([0,K(0,0),1,\ldots,1]) = [k_0,1,1,\ldots,1,k_{N-1},1]
            \end{equation*}
            for some constants $k_0$ and $k_{N-1}$ with $k_{N-1} \not\in \{0,1\}$ (again we may need to exclude finitely many choices of $K(1,1)$ and $K(0,N-2)$ so that $k_{N-1} \not\in \{0,1\}$).  In particular, since the $c_i(0,0)$ are already chosen, each \\
            $x_i(\phi([k_0,1,1,\ldots,1,k_{N-1},1]))$ is linear in $c_i(N-1,N-1)$.
            From our choice of $[0,K(0,0),1,\ldots,1]$ we have determined $\phi([0,K(0,0),1,\ldots,1])$ and have $\phi(\phi([0,K(0,0),1,\ldots,1]))$ as our next iterate to consider.  We have thus increased the primitive period of $[0,\ldots,0,1]$ by two with the choice of the $c_i(0,0)$.

            If $N=3$ we are done since we are linear in the $c_i(N-1,N-1)$ and they are the only unchosen coefficients; so we choose the $c_i(N-1,N-1)$ to make the point periodic.

            For $N > 3$ we repeat the process.  Choose the $c_i(N-1,N-1)$ to get
            \begin{equation*}
                [0,0,K(N-1,N-1),1,\ldots,1,0,1]
            \end{equation*}
            with $K(N-1,N-1) \not\in \{0,1\}$.
            The coordinates of the image are linear in the $c_i(2,2)$ so we choose
            \begin{equation*}
                \phi([0,0,K(N-1,N-1),1,\ldots,1,0,1]) = [0,0,K(2,2),1,\ldots,1,0,1],
            \end{equation*}
            for $K(2,2) \not\in \{0,1,K(N-1,N-1)\}$.  The image \\
            $\phi([0,0,K(2,2),1,\ldots,1,0,1])$ is completely determined since we have three points of the form
            $[0,0,x_2,1,\ldots,1,0,1]$ whose images are known.  These images are
            \begin{align*}
                [0,0,K(N-1,N-1),1,\ldots,1,0,1] &\xrightarrow{\phi} [0,0,K(2,2),1,\ldots,1,0,1]\\
                [0,0,1,1,\ldots,1,0,1] &\xrightarrow{\phi} [0,0,0,1,\ldots,1,1,1,1]\\
                [0,0,0,1,\ldots,1,0,1] &\xrightarrow{\phi} [0,0,0,0,1,\ldots,1,1,1,1].
            \end{align*}
            We have
            \begin{equation*}
                \phi([0,0,K(2,2),1,\ldots,1,0,1]) = [0,0,x,y,1,\ldots,1,z,1]
            \end{equation*}
            for some constants $x$, $y$, and $z$.  Note that we may need to exclude finitely many choices of $K(N-1,N-1)$ and $K(2,2)$ so that $y \not\in \{0,1\}$.  We have already chosen the $c_i(2,2)$ and the $c_i(N-1,N-1)$, so each \\
            $x_i(\phi(\phi([0,0,K(2,2),1,\ldots,1,0,1])))$ is linear in $c_i(3,3)$, again increasing the primitive period by two with the choice of a single set of coefficients $c_i(2,2)$.

            Continuing in this manner, we choose the $c_i(k,k)$ to get
            \begin{equation*}
                [0,\ldots,0,K(k,k),1,\ldots,1,0,1],
            \end{equation*}
            where $K(k,k)\not\in \{0,1\}$ and is the $(k+1)\sst$ coordinate.  The $i\tth$ coordinate of the image is linear in $c_i(k+1,k+1)$.  We choose
            \begin{equation*}
                \phi([0,\ldots,0,K(k,k),1,\ldots,1,0,1]) = [0,\ldots,0,K(k+1,k+1),1,\ldots,1,0,1],
            \end{equation*}
            where $K(k+1,k+1)$ is the $(k+1)\sst$ coordinate and $K(k+1,k+1) \not\in \{0,1,K(k,k)\}$.
            The image $\phi([0,\ldots,0,K(k+1,k+1),1,\ldots,1,0,1])$ is completely determined since we have three points of the form \\
             $[0,\ldots,0,x_{k+1},1,\ldots,1,0,1]$ whose images are known; they are
            \begin{align*}
                [0,\ldots,0,K(k,k),1,\ldots,1,0,1] &\xrightarrow{\phi} [0,\ldots,0,K(k+1,k+1),1,\ldots,1,0,1]\\
                [0,\ldots,0,1,1,\ldots,1,0,1] &\xrightarrow{\phi} [0,\ldots,0,0,1,\ldots,1,1,1,1]\\
                [0,\ldots,0,0,1,\ldots,1,0,1] &\xrightarrow{\phi} [0,\ldots,0,0,0,1,\ldots,1,1,1,1].
            \end{align*}
            We have
            \begin{equation*}
                \phi([0,\ldots,0,K(k+1,k+1),1,\ldots,1,0,1]) = [0,\ldots,0,x,y,1,\ldots,1,z,1]
            \end{equation*}
            for some constants $x$, $y$, and $z$.  Note that we may need to exclude finitely many choices of $K(k,k)$ and $K(k+1,k+1)$ so that $y \not\in \{0,1\}$.  We have already chosen the $c_i(k+1,k+1)$ and the $c_i(N-1,N-1)$ so each $x_i(\phi(\phi([0,\ldots,0,K(k+1,k+1),1,\ldots,1,0,1])))$ is linear in $c_i(k+2,k+2)$, again increasing the primitive period by $2$.

            We continue this process until the only non-chosen coefficients are either $\{c_i(N-2,N-2),c_i(N-3,N-3)\}$ or $\{c_i(N-2,N-2)\}$. In this first case, we do not have enough unchosen coefficients remaining to increase the primitive period further beyond the trivial value, so we simply choose the $c_i(N-3,N-3)$ to have the point
            \begin{equation*}
                [0,\ldots,0,K(N-3,N-3),1,1]
            \end{equation*}
            with $K(N-3,N-3) \not\in \{0,1\}$.  Each $x_i(\phi([0,\ldots,0,K(N-3,N-3),1,1]))$ is linear in $c_i(N-2,N-2)$.
            We have now reduced to the second case and choose the $c_i(N-2,N-2)$ to have
            \begin{equation*}
                \phi([0,\ldots,0,K(N-3,N-3),1,1])=[0,0,\ldots,0,1],
            \end{equation*}
            making the point periodic of primitive period $\frac{(N+1)(N+2)}{2} + \lfloor \frac{N-1}{2} \rfloor$.

            Note, that along the way we were able to choose $$\{c_1(0,N-1),c_1(0,0),c_0(1,1),\ldots,c_{N-2}(N-3,N-3), c_2(N-1,N-1)\}$$ arbitrarily, except for excluding a finite set of values, making this an infinite family of dimension $N$.

            It is easy to modify this construction to get points with periods \\
            $< \frac{(N+1)(N+2)}{2} + \left\lfloor \frac{N-1}{2} \right\rfloor$ since at each stage we are linear in at most two variables.  So we simply choose the coefficients so that $\phi^n(P) = [0,\ldots,0,1]$ at the appropriate iterate.  The dimension of these families is larger since there are more free coefficients.
        \end{case}
    \end{proof}

    \section{Proof of Theorem \ref{thm2}}
        We will use the theory of Macaulay resultants to show that we can choose the coefficients of the maps in Theorem \ref{thm_main} so that they are morphisms; in other words, so that $\phi_0,\ldots,\phi_N$ have no nontrivial common zeroes.  Following \cite{Macaulay}: given $N+1$ homogeneous forms $F_0,\ldots,F_N$ of degree $d_i$ in $N+1$ variables $x_0,\ldots,x_N$, construct a matrix denoted $M_d(F_0,\ldots,F_N)$ where $d=1+\sum_i (d_i-1)$.  The columns of $M_d$ correspond to the monomials of degree $d$ in the variables $x_0,\ldots,x_N$, and the rows correspond to polynomials of the form $rF_i$ where $r$ is a monomial such that $\deg(rF_i)=d$.  The entries of $M_d$ are the coefficients of the column monomials in the row polynomials.  The matrix has $\binom{N+d}{d}$ columns and the number of rows corresponding to each $F_i$ is $\binom{N+d-d_i}{d-d_i}$.  It is the transpose of the matrix of the linear map
        \begin{equation*}
            (P_0,\ldots,P_N) \mapsto P_0F_0 + \cdots + P_NF_N,
        \end{equation*}
        where $P_i$ is homogenous of degree $d-d_i$.  Consider the maximal minors of $M_d(F_0,\ldots,F_N)$.  The determinants of these minors are polynomials in the coefficients of $F_0,\ldots,F_N$.  Let $R$ be the greatest common divisor of these determinants (as polynomials in the coefficients).  Then $R$ is called the \emph{resultant} of $F_0,\ldots,F_N$ and (among other properties) it satisfies $R=0$ if and only if the forms $F_0,\ldots,F_N$ have a common nontrivial zero.

        \begin{proof}[Proof of Theorem \ref{thm2}]
            We are in the case of $N+1$ homogeneous forms $\phi_i$ in $N+1$ variables $x_0,\ldots,x_N$.  We have each $\phi_i$ is degree $2$ and hence $d=N+2$.  We will show that the Macaulay matrix has a maximal minor that has nonzero determinant and hence that the resultant is nonzero for infinitely many maps in the family.  In the matrix there are $\binom{2N+2}{N+2}$ columns corresponding to all of the monomials with degree $N+2$ and $(N+1)\binom{2N}{N}$ rows corresponding to the $\binom{2N}{N}$ monomials of degree $d-2$ for each of the $N+1$ forms $\phi_i$.  We need to extract a $\binom{2N+2}{N+2} \times \binom{2N+2}{N+2}$ minor with nonzero determinant.  We first consider the case of largest possible period from Theorem \ref{thm_main}.

            For $N=2$ we can write down the matrix (but do not do so here) and explicitly check that it has a maximal minor with nonzero determinant.

            For $N \geq 3$ define (with $N-2$ replaced with $N-1$ for $N=3$)
            \begin{align*}
                S_N &= \{ F \col \deg(F) = d, x^2_N \mid F\}\\
                S_{N-2} &= \{F \col \deg(F) = d, x^2_N \nmid F, \text{ and } x^2_{N-2} \mid F\}\\
                S_{N-1} &= \{F \col \deg(F) = d, x^2_N \nmid F,x^2_{N-2} \nmid F, \text{ and } x^2_{N-1} \mid F\}\\
                S_{N-3} &= \{F \col \deg(F) = d, x^2_N \nmid F, x_{N-1}^2 \nmid F, x^2_{N-2} \nmid F, \text{ and } x^2_{N-3} \mid F\}\\
                S_{N-4} &= \{F \col \deg(F) = d, x^2_N \nmid F, \ldots ,x^2_{N-3} \nmid F, \text{ and } x^2_{N-4} \mid F\}\\
                &\vdots\\
                S_0 &= \{F \col \deg(F) = d, x^2_N \nmid F, \ldots, x^2_1\nmid F, \text{ and } x^2_{0} \mid F\}.
            \end{align*}
            Order the columns in reverse lexicographic order, $x_N > x_{N-1} > \cdots > x_0$, with the largest to the left.  For the columns corresponding to a monomial in $S_N$, choose the row with a $1$ on the diagonal (the row contains all $0$'s except one entry which is $1$ since $\phi_N(x_0,\ldots,x_N) = x_N^2$).  For columns corresponding to monomials in $S_{2k}$ with $k \neq 0$, choose the row with $c_{2k}(2k,2k)$ on the diagonal.  For columns corresponding to monomials in $S_{2k-1}$ with $k\neq 1$, choose the row with $c_{2k}(2k-1,2k-1)$ on the diagonal.  For $S_1$ we choose the row with $c_0(1,1)$ on the diagonal, and for $S_0$ we choose the row with $c_1(0,0)$ on the diagonal.  Finally, columns corresponding to monomials in $S_{N-2}$ we fix $i>1$ odd and choose the row with $c_i(N-2,N-2)$ on the diagonal (use $S_{N-1}$ and $c_i(N-1,N-1)$ for $N=3$).

            We have two facts to verify:
            \begin{enumerate}
                \item These choices contain no duplicate rows.
                \item The resulting minor has nonzero determinant.
            \end{enumerate}
            The first is clear since $S_i$ and $S_j$ are disjoint for $i \neq j$ and each row associated to an element of $S_k$ has at most one entry containing a $c_i(k,k)$.

            For the second, we start by examining the entries in each row.  Each row associated to $\phi_i$ for $i \neq N$ contains a $c_i(N-2,N-2)$ (or $c_i(N-1,N-1)$ for $N=3$) whose value depends on at least $c_i(N-3,N-3)$ (or $c_i(1,1)$ for $N=3$) so is not identically $0$.  In addition, each of these rows contains a corresponding
            \begin{equation*}
                c_i(N-2,N) = \begin{cases}
                    -1-c_i(N-2,N-2) & \text{for } i=0\\
                    1 - c_i(N-2,N-2) & \text{for } i = N-2\\
                    -c_i(N-2,N-2) & \text{otherwise}.
                \end{cases}
            \end{equation*}
            For $x_1^2$ we have $c_0(1,1)$ and $c_0(1,N)=-1-c_0(1,1)$.  For $x_0^2$, we have $c_1(0,0)$ and $c_1(0,N)=-c_0(1,1)$.  For $x_{2k}^2$ with $k>0$, there is a $c_{2k}(2k,2k)$ and a $c_{2k}(2k,N) =1 - c_{2k}(2k,2k)$.  For $k>1$ there is a $c_{2k}(2k-1,2k-1)$ and a $c_{2k}(2k-1,N)=1-c_{2k}(2k-1,2k-1)$.  The rest of the entries are either constants or depend on $c_{1}(0,N-1)$.  Also note that each row contains each $c_i(k,k)$ at most once (in addition to the corresponding $c_i(k,N)$).

            Note that $c_i(k,k)$ and $c_i(k,N)$ are possibly linearly dependent and that our choice of ordering has $c_i(k,k)$ appearing farther right in the matrix than $c_i(k,N)$.  For the other entries, we are choosing one coefficient per iteration, so they are either constant, independent, or the next depends on the previous in a quadratic (or higher) fashion (since each $\phi_i$ is degree $2$).

            Assume that we have some linear combination of the rows that produces a row identically $0$.   Each row contains a $c_i(k,k)$ on the diagonal for some $i$ and $k$ and a $c_i(k,N)$ in some other entry.  For the linear combination to result in $0$, there are three cases to consider.
            \setcounter{case}{0}
            \begin{case} \label{case1}
                  Assume two rows in the combination contain $c_i(k,k)$ and $c_i(k,N)$ in the same column.  By our choice of ordering, the respective $c_i(k,N)$ and $c_i(k,k)$ in those rows would not be in the same column.  Hence, we must also include rows in the combination that contain $c_i(k,N)$ and $c_i(k,k)$ in the corresponding columns.  Again by our choice of ordering, we need to include at least two rows to do this and then we still have unpaired $c_i(k,k)$ and $c_i(k,N)$ as before.  Therefore, we cannot choose any number of rows so that all of the $c_i(k,k)$ and $c_i(k,N)$ are paired by column.
            \end{case}
            \begin{case} \label{case2}
                Notice that by our choice of the $S_i$ we have guaranteed that we cannot have $c_j(N-2,N-2)$ and $c_i(k,k)$ in the same column for any $k \neq N-2$.  Let $j$ be such that $c_j(N-2,N-2)$ is on the diagonal of the minor.  Assume two rows in the combination have $c_i(N-2,N-2)$ and $c_j(N-2,N-2)$ in the same column for $i \neq j$.  But with $c_j(N-2,N-2)$ used for $S_{N-2}$, we must have that the row containing $c_i(N-2,N-2)$ also contains $c_i(k,k)$ for some $k \neq N-2$.  As in Case \ref{case1}, we are unable to find a combination of rows that pairs all of the $c_i(k,k)$ and $c_i(k,N)$.
            \end{case}
            \begin{case}
                Assume we have $c_i(k,k)$ and $c_i(k,N)$ paired with constants to get a combination of rows identically $0$.  However, every row containing a $c_i(k,k)$ with $k\neq N-2$ also contains $c_j(N-2,N-2)$ for some $j$.  These $c_j(N-2,N-2)$ must be paired either with constants or with other $c_t(N-2,N-2)$.  However, they cannot be paired with constants since the $c_i(k,k)$ are already paired with constants in a combination that results in $0$, and $c_i(k,k)$ and $c_j(N-2,N-2)$ are not related in a linear fashion.  Case \ref{case2} eliminates the possibility of pairing with another $c_t(N-2,N-2)$ for some $t$.  So we must have $k = N-2$.  Then all of the rows in the combination are associated to the same $\phi_j$ and, hence, entries in columns cannot be paired appropriately to result in a combination of $0$.
            \end{case}
            Therefore, no linear combination can have all entries as $0$ and the determinant of this minor is not identically $0$.  Therefore, there are infinitely many choices of the coefficients that produce a map that is a morphism.

            For the families with a periodic point with smaller primitive period, the matrix is similar but with more free constants, so similar choices of rows will also produce a minor with nonzero determinant.
    \end{proof}

\section{Proof of Theorem \ref{thm3}}

\begin{lem} \label{lem2}
    Given $\phi_1: \P^N \to \P^N$ a polynomial morphism with $P_1$ of primitive period $n$ and $\phi_2:\P^M \to \P^M$ a polynomial morphism with $P_2$ of primitive period $m$, then there exists a polynomial morphism $\psi: \P^{N+M} \to \P^{N+M}$ and a point $P$ with primitive period $\lcm(n,m)$.
\end{lem}
\begin{proof}
    We restrict $\phi_1$ to the affine chart $\A^N$ with $x_N \neq 0$ and $\phi_2$ to the affine chart $\A^M$ with $x_M \neq 0$.  The restricted points $\widetilde{P_1}$ and $\widetilde{P_2}$ still have period $n$ and $m$, and the product map $\widetilde{\phi_1} \times \widetilde{\phi_2}: \A^{N+M} \to \A^{N+M}$ has the product of the dehomogenizations $\widetilde{P}=(\widetilde{P_1},\widetilde{P_2})$ as a periodic point of primitive period $\lcm(n,m)$.  This fact is simply the statement that the product of a cyclic group of order $n$ with a cyclic group of order $m$ has order $\lcm(n,m)$.

    Now, homogenizing $\widetilde{\phi_1} \times \widetilde{\phi_2}$ to a map $\psi:\P^{N+M} \to \P^{N+M}$ we know that the first $N$ forms and $x_{N+M}^2$ have no common nontrivial zeros in $x_0,\ldots,x_{N-1},x_{N+M}$ and the next $M$ forms and $x_{N+M}^2$ have no common nontrivial zeros in $x_N,\ldots,x_{N+M}$. Since the only variable shared between the two sets of forms is $x_{N+M}$, the map $\psi$ is also a morphism and the homogenization of $\widetilde{P}$ has primitive period $\lcm(n,m)$ for $\psi$.
\end{proof}

\begin{proof}[Proof of Theorem \ref{thm3}]
    From Theorem \ref{thm_main} and Theorem \ref{thm2} we can find morphisms $\phi:\P^N \to \P^N$ with $\Q$-rational periodic points with primitive period $1,2,\ldots, \frac{(N+1)(N+2)}{2}$.  Fix $s$ a positive integer.  Let $M=\lfloor N/s \rfloor$. Then $\frac{(M+1)(M+2)}{2} > \frac{(N/s)(N/s)}{2} = \frac{N^2}{2s^2}$ and for every prime $p \leq \frac{N^2}{2s^2}$ there is a point with primitive period $p$ for some polynomial morphism of $\P^M$. Fix $\epsilon >0$ and choose $N$ large enough that the interval $((1-\epsilon)N^2/2s^2, N^2/2s^2)$ has at least $s$ primes $p_1,\ldots, p_s$.  We apply Lemma \ref{lem2} to combine these points and associated morphisms to get a point $P \in \P^{sM} = \P^{N}$, which has primitive period
    \begin{equation*}
        p_1\cdots p_s \geq \frac{(1-\epsilon)^s}{2^2s^{2s}}N^{2s}
    \end{equation*}
    for a polynomial morphism $\psi:\P^{N} \to \P^{N}$.
\end{proof}

\section{Some examples with larger primitive periods}
    With slightly different choices of coefficients, it is occasionally possible to increase the primitive period by more than $2$ with a choice of a single set of coefficients.  While a general method to ensure this occurrence was not discovered, in practice it is possible to construct a polynomial map for a specific $N$ with a periodic point with primitive period that exceeds the bound presented in Theorem \ref{thm_main}.  These can then be combined as in Lemma \ref{lem2} to produce morphisms of $\P^N$ with $\Q$-rational periodic points of large primitive period.  The following examples present such maps for $N=2$, $3$, and $4$.
    For the reader's convenience, the following table shows the trivial lower bound $\binom{N+2}{2}$, the lower bound from Theorem \ref{thm_main}, and the primitive period exhibited in the example of a polynomial morphism $\phi:\P^N \to \P^N$.  Note that since we are dealing with maps on $\P^N$ outside of the scope of Theorem \ref{thm2}, the maps were verified explicitly to be morphisms, but the details are omitted here.
    \begin{center}
      \begin{tabular}{|c|c|c|c|}
        \hline
        N & trivial bound & Theorem \ref{thm_main} bound & example period\\
        \hline
        2 & 6 & 7 & 9\\
        \hline
        3 & 10 & 11 & 24\\
        \hline
        4 & 15 & 16 & 72\\
        \hline
      \end{tabular}
    \end{center}

    \begin{exmp}\label{exmp1}
        The point $[0,0,1] \in \P^2$ is a periodic point of primitive period $9$ for the morphism
        \begin{align*}
            \phi([&x_0,x_1,x_2])=\\
                [&-38/45x_0^2 + (2x_1 - 7/45x_2)x_0 + (-1/2x_1^2 - 1/2x_2x_1 + x_2^2),\\
                &-67/90x_0^2 + (2x_1 + 157/90x_2)x_0 - x_2x_1,x_2^2].\\
        \end{align*}
    \end{exmp}
    \begin{exmp}
        The point $[0,0,0,1] \in \P^3$ is a periodic point of primitive period $24$ for the morphism
        \begin{align*}
            \phi([&x_0,x_1,x_2,x_3])=\\
            [&(-x_1 - x_3)x_0 + (-13/30x_1^2 + 13/30x_3x_1 + x_3^2),\\
            & -1/2x_0^2 + (-x_1 + 3/2x_3)x_0 + (-1/3x_1^2 + 4/3x_3x_1),\\
            &-3/2x_2^2 + 5/2x_2x_3 + x_3^2,x_3^2]
        \end{align*}
        created by combining a periodic point of primitive period $8$ in $\P^2$ and a periodic point of primitive period $3$ in $\P^1$.
    \end{exmp}

    \begin{exmp}
        The point $[0,0,0,0,1] \in \P^4$ is a periodic point of primitive period $72$ for the morphism
        \begin{align*}
            \phi([&x_0,x_1,x_2,x_3,x_4])=\\
                [&-38/45x_0^2 + (2x_1 - 7/45x_4)x_0 + (-1/2x_1^2 - 1/2x_4x_1 + x_4^2),\\
                &-67/90x_0^2 + (2x_1 + 157/90x_4)x_0 - x_4x_1,\\
                &(-x_3 - x_4)x_2 + (-13/30x_3^2 + 13/30x_4x_3 + x_4^2),\\
                 &-1/2x_2^2 + (-x_3 + 3/2x_4)x_2 + (-1/3x_3^2 + 4/3x_4x_3),x_4^2].\\
        \end{align*}
        created by combining periodic points of primitive period $8$ and $9$ in $\P^2$.
    \end{exmp}

\bibliographystyle{abbrv}	
\bibliography{masterlist}

\end{document}